\documentclass[11pt]{amsart}
\usepackage[utf8]{inputenc}
\usepackage[T1]{fontenc}
\usepackage{amsfonts}
\usepackage{amssymb}
\usepackage{color}
\usepackage{caption}
\usepackage{caption,subcaption}
\usepackage[labelfont=bf]{caption}
\usepackage{graphicx,psfrag,rotating}
\usepackage{amsmath}
\usepackage{mathptmx}
\usepackage{relsize}

\newcommand\Z{\mathbb{Z}}

\newcommand\R{\mathbb{R}}

\newcommand\cF{{\mathcal F}}

\newcommand\cV{{\mathcal V}}

\newcommand\cL{{\mathcal L}}

\allowdisplaybreaks[1]
\theoremstyle{plain}

\newtheorem{theorem}{Theorem}[section]
\newtheorem{lemma}[theorem]{Lemma}

\newtheorem*{main thm}{Main Theorem}

\theoremstyle{definition}
\newtheorem{definition}[theorem]{Definition}
\newtheorem{example}[theorem]{Example}

\newtheorem{remark}[theorem]{Remark}

\theoremstyle{remark}

\begin{document}

\newenvironment{prooff}{\medskip \par \noindent {\it Proof}\ }{\hfill
$\square$ \medskip \par}
    \def\sqr#1#2{{\vcenter{\hrule height.#2pt
        \hbox{\vrule width.#2pt height#1pt \kern#1pt
            \vrule width.#2pt}\hrule height.#2pt}}}
    \def\square{\mathchoice\sqr67\sqr67\sqr{2.1}6\sqr{1.5}6}
\def\pf#1{\medskip \par \noindent {\it #1.}\ }
\def\endpf{\hfill $\square$ \medskip \par}
\def\demo#1{\medskip \par \noindent {\it #1.}\ }
\def\enddemo{\medskip \par}
\def\qed{~\hfill$\square$}

%%%%%%%%%%%%%%%%%%%%%%%%%%%%%%%%%%%%%%%%%%%%%%%%%%%%%%%%%%%%%%%%%%%%%%%%%%%%%%%%%%%%%%%%%%%%%%%%%%%%%%%%%%%%%%%%%%%%%%%%%%%%%%%%%%

\title[Dynnikov Coordinates]
{Dynnikov Coordinates on Punctured Torus}

\author[A. Meral]
{Alev Meral}


\address{(A. M.) Department of Mathematics, D\.{ı}cle University, 21280
D\.{ı}yarbak{\i}r, Turkey}

%\email{alev.meral@dicle.edu.tr}
\email{$alev.meral@dicle.edu.tr$}

\date{\today}

%%%%%%%%%%%%%%%%%%%%%%%%%%%%%%%%%%%%%%%%%%%%%%%%%%%%%%%%%%%%%%%%%%%%%%%%%%%%%%%%%%%%%%%%%%%%%%%%%%%%%%%%%%%%%%%%%%%%%%%%%%%%%%%%%%
\begin{abstract}
We generalize Dynnikov coordinate system previosly defined on the standard punctured disk to an orientable surface of genus-$1$ with $n$ punctures and one boundary component. 
\end{abstract}

\maketitle

\setcounter{secnumdepth}{2}
\setcounter{section}{0}

%%%%%%%%%%%%%%%%%%%%%%%%%%%%%%%%%%%%%%%%%%%%%%%%%%%%%%%%%%%%%%%%%%%%%%%%%%%%%%%%%%%%%%%%%%%%%%%%%%%%%%%%%%%%%%%%%%%%%%%%%%%%%%%%%%
\section{Introduction}

The aim of this paper is to generalize Dynnikov coordinates to a genus-$1$  surface with $n ~(n \geq 2)$ punctures and one boundary component.
\emph{Dynnikov coordinates} \cite{dynnikov02} is an effective way to coordinatize an integral lamination on a finitely punctured disk $D_n$  ~$(n \geq 3)$. It provides a bijection between the isotopy classes of integral laminations and $\Z^{2n-4} \setminus \{0\}$. Dynnikov coordinate system has been extensively used to solve various dynamical and combinatorial problems such as word problem in the braid group \cite{dehornoy02}, \cite{dehornoy08},  calculating the topological entropies of pseudo-Anosov braids \cite{yurttas11}, \cite{moussafir06} and computing the geometric intersection number of two integral laminations on $D_n$ \cite{yurttas18}.

 Throughout the paper, $S_n$ will denote a genus-$1$  surface with $n ~(n \geq 2)$ punctures and one boundary component . To coordinatize a given integral lamination on $S_n$, a system consisting of $3n+2$ arcs and a simple  closed curve on $S_n$ is used. Given an integral lamination $\cL$ (or a measured foliation $\cF$), at first we have introduced a vector in $\Z^{3n+3}_{\geq 0} \setminus \{0\}$ (or $\R^{3n+3}_{\geq 0} \setminus \{0\}$) using geometric intersection numbers  (or the measure assigned to these curves) with the curves in our system.  To uniquely determine every lamination we  have defined  Dynnikov coordinates on $S_n$  by considering linear combinations of these intersection numbers (see Section~\ref{our_coordinates}).

%%%%%%%%%%%%%%%%%%%%%%%%%%%%%%%%%%%%%%%%%%%%%%%%%%%%%%%%%%%%%%%%%%%%%%%%%%%%%%%%%%%%%%%%%%%%%%%%%%%%%%%%%%%%%%%%%%%%%%%%%%%%%%%%%%
\section{Dynnikov Coordinates on $S_n$}\label{our_coordinates}
In this section, we describe Dynnikov coordinates  on  $S_n$. For this, we use the model shown in Figure~\ref{model}. 
\begin{figure}[!ht]
\centering
\psfrag{a1}[tl]{\scalebox{0.6}{$\scriptstyle{\alpha_{1}}$}}
\psfrag{a2}[tl]{\scalebox{0.6}{$\scriptstyle{\alpha_{2}}$}} 
\psfrag{a(2i-3)}[tl]{\scalebox{0.6}{$\scriptstyle{\alpha_{2i-3}}$}} 
\psfrag{a(2i-2)}[tl]{\scalebox{0.6}{$\scriptstyle{\alpha_{2i-2}}$}} 
\psfrag{a(2i-1)}[tl]{\scalebox{0.6}{$\scriptstyle{\alpha_{2i-1}}$}}
\psfrag{a(2i)}[tl]{\scalebox{0.6}{$\scriptstyle{\alpha_{2i}}$}} 
\psfrag{a(2i+1)}[tl]{\scalebox{0.6}{$\scriptstyle{\alpha_{2i+1}}$}} 
\psfrag{a(2i+2)}[tl]{\scalebox{0.6}{$\scriptstyle{\alpha_{2i+2}}$}} 
\psfrag{a(2n-1)}[tl]{\scalebox{0.6}{$\scriptstyle{\alpha_{2n-1}}$}}
\psfrag{a(2n)}[tl]{\scalebox{0.6}{$\scriptstyle{\alpha_{2n}}$}} 
\psfrag{b1}[tl]{\scalebox{0.6}{$\scriptstyle{\beta_{1}}$}} 
\psfrag{bi}[tl]{\scalebox{0.6}{$\scriptstyle{\beta_{i}}$}} 
\psfrag{b(i+1)}[tl]{\scalebox{0.6}{$\scriptstyle{\beta_{i+1}}$}} 
\psfrag{b(n+1)}[tl]{\scalebox{0.6}{$\scriptstyle{\beta_{n+1}}$}} 
\psfrag{c}[tl]{\scalebox{0.6}{$\scriptstyle{c}$}}
\psfrag{g}[tl]{\scalebox{0.6}{$\scriptstyle{\gamma}$}}
\psfrag{d(2i-1)}[tl]{\scalebox{0.6}{$\scriptstyle{{\color{blue}\Delta_{2i-1}}}$}}
\psfrag{d(2i)}[tl]{\scalebox{0.6}{$\scriptstyle{{\color{blue}\Delta_{2i}}}$}}
\psfrag{d(2i-2)}[tl]{\scalebox{0.6}{$\scriptstyle{{\color{blue}\Delta_{2i-2}}}$}}
\psfrag{d(2i+1)}[tl]{\scalebox{0.6}{$\scriptstyle{{\color{blue}\Delta_{2i+1}}}$}}
\psfrag{d(2n)}[tl]{\scalebox{0.6}{$\scriptstyle{{\color{blue}\Delta_{2n}}}$}}
\psfrag{d1}[tl]{\scalebox{0.6}{$\scriptstyle{{\color{blue}\Delta_{1}}}$}}
%{\scalebox{0.6}{\includegraphics{ucgen_koord}}}
\includegraphics[width=0.65\textwidth]{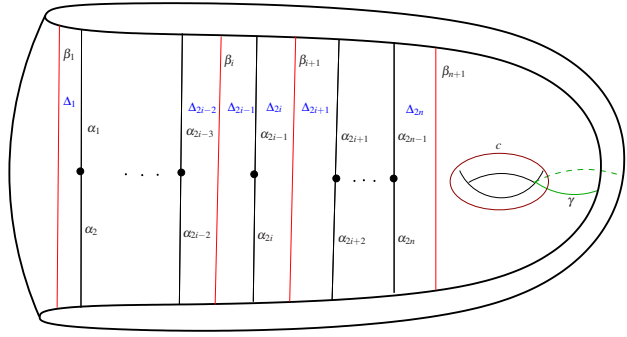}
\caption{ Curves on $S_n$ }\label{model}
\end{figure} 
Here, the arcs $\alpha_i  ~(1 \leq i \leq 2n)$ and $\beta_i  ~(1 \leq i \leq n+1)$ are similar to the $D_n$ case. That is, the end points of these arcs are either on the boundary or on the puncture. While $c$ is the longitude of the torus, $\gamma$ is the arc whose both end points are on the boundary. Also, note that $\gamma$ intersects with $c$ once transversally.

Let  $\cL_n$ be the set of integral laminations on $S_n$ and $\cL \in \cL_n$.  Throuhgout the paper, we always work with the minimal representative (an integral lamination in the same isotopy class intersecting with coordinate curves minimally) of $\cL$ and denote it by $L$. Let the vector $(\alpha_{1}, \cdots, \alpha_{2n}; \beta_{1}, \cdots, \beta_{n+1}; \gamma; c) \in \{\Z^{3n+3}_{\geq 0}\} \setminus \{0\}$ show the intersection numbers of $L$ with the corresponding arcs and the simple closed curve $c$. For example, $(4, 1, 3, 2, 4, 1; 3, 5, 5, 3; 3; 1)$ are the intersection numbers of the integral lamination $L$ depicted in Figure~\ref{ucgen_ornek}.
\begin{figure}[!ht]
\centering
%{\scalebox{0.6}{\includegraphics{ucgen_ornek}}}
\includegraphics[width=0.65\textwidth]{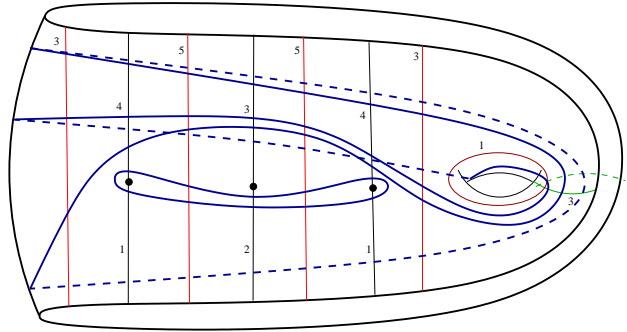}
\caption{ Intersection numbers with coordinate curves }\label{ucgen_ornek}
\end{figure}

If $L$ contains $p(c)$ many copies of $c$, then let
\begin{equation}\label{copy}
c = -p(c),
\end{equation}
 \noindent where $p(c) > 0$. Throughout the paper we define $c^{+}$ as  $\max(c, 0)$.

\subsection{Path Components on $S_n$}
In this section, we are going to introduce path components of an integral lamination $L$ on $S_{n}$ and derive formulas for the number of these components.

Let $U_i  ~(1 \leq i \leq n)$ be the region that is bounded by $\beta_i$ and $\beta_{i+1}$  (Figure~\ref{puncture_bilesen}) and $G$ be the region bounded by $\beta_1$, $\beta_{n+1}$ and the boundary of $S_n$ ~($\partial S_n$) (Figures ~\ref{genus_loop} and \ref{cikis_yon}). Since $L$ is minimal, there are $4$ types of path components in $U_i$ ~$(1 \leq i \leq n)$  as on the disk \cite{yurttas13}: \emph{Above component}; which has end points on $\beta_i$ and $\beta_{i+1}$ intersecting with $\alpha_{2i-1}$, \emph{below component}; which has end points on $\beta_i$ and $\beta_{i+1}$ intersecting with $\alpha_{2i}$, \emph{left loop component}; which has both end points  on $\beta_{i+1}$ intersecting with $\alpha_{2i-1}$ and $\alpha_{2i}$ (Figure~\ref{puncture_bilesen} (a)) and \emph{right loop component}; which has both end points are on $\beta_i$ intersecting  with $\alpha_{2i-1}$ and $\alpha_{2i}$ (Figure~\ref{puncture_bilesen} (b)).  There are $6$ types of path components in $G$. First three are \emph{$c$ curve}; bounding the genus of surface (Figure~\ref{genus_loop} (a)), \emph{front genus component}; which has both end points  on $\beta_{n+1}$  not intersecting with curve $c$ (Figure~\ref{genus_loop} (b)), \emph{back genus component}; having both endpoints on $\beta_{1}$ not intersecting with curve $c$  (Figure~\ref{genus_loop} (c)).  The other three components, called \emph{twisting} which have end points on $\beta_{1}$ and $\beta_{n+1}$ intersecting with curve $c$ (see Figure~\ref{cikis_yon}) are \emph{non-twist  component}; see Figure~\ref{cikis_yon} (a), \emph{negative twist component} which makes clockwise twist (See Figure~\ref{cikis_yon} (b)), and   \emph{positive twist  component} which makes counterclockwise twist (See Figure~\ref{cikis_yon} ~(c)).

\begin{figure}[!ht]
\centering
\psfrag{a}[tl]{\scalebox{0.6}{$\scriptstyle{a}$}}
\psfrag{b}[tl]{\scalebox{0.6}{$\scriptstyle{b}$}} 
\psfrag{a2i}[tl]{\scalebox{0.6}{$\scriptstyle{\alpha_{2i}}$}} 
\psfrag{a2i-1}[tl]{\scalebox{0.6}{$\scriptstyle{\alpha_{2i-1}}$}}
\psfrag{bi}[tl]{\scalebox{0.6}{$\scriptstyle{\beta_{i}}$}} 
\psfrag{bi+1}[tl]{\scalebox{0.6}{$\scriptstyle{\beta_{i+1}}$}} 
\psfrag{d(2i-1)}[tl]{\scalebox{0.6}{$\scriptstyle{{\color{blue}\Delta_{2i-1}}}$}}
\psfrag{d(2i)}[tl]{\scalebox{0.6}{$\scriptstyle{{\color{blue}\Delta_{2i}}}$}}
\psfrag{U(i)}[tl]{\scalebox{0.6}{$\scriptstyle{U_{i}}$}}
%{\scalebox{0.7}{\includegraphics{puncture_bilesen}}}
\includegraphics[width=0.65\textwidth]{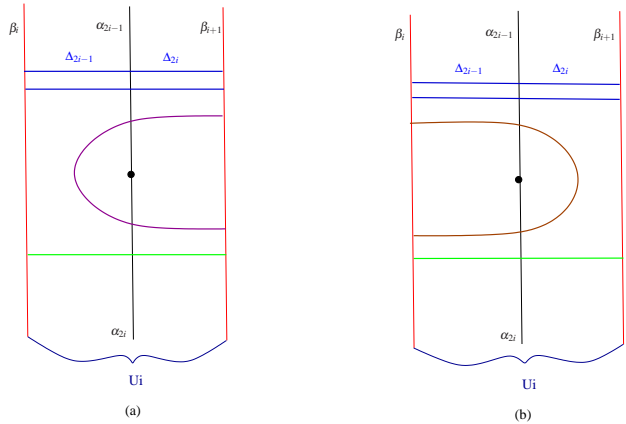}
\caption{Above and below components, left and right loop components in region $U_i$}\label{puncture_bilesen}
\end{figure}

\begin{figure}[!ht]
\centering
\psfrag{a}[tl]{\scalebox{0.55}{$\scriptstyle{a}$}}
\psfrag{b}[tl]{\scalebox{0.55}{$\scriptstyle{b}$}} 
\psfrag{g}[tl]{\scalebox{0.6}{$\scriptstyle{\gamma}$}}
\psfrag{t}[tl]{\scalebox{0.6}{$\scriptstyle{c}$}} 
\psfrag{bn+1}[tl]{\scalebox{0.6}{$\scriptstyle{\beta_{n+1}}$}}
\psfrag{b1}[tl]{\scalebox{0.6}{$\scriptstyle{\beta_{1}}$}} 
%{\scalebox{0.9}{\includegraphics{genus_loop}}}
\includegraphics[width=0.9\textwidth]{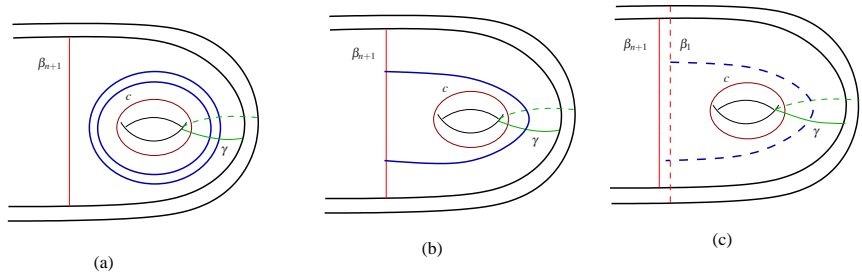}
\caption{(a) $c$ curves, (b) front genus component, (c) back genus component in region $G$ }\label{genus_loop}
\end{figure}

\begin{figure}[!ht]
\centering
\psfrag{a}[tl]{\scalebox{0.55}{$\scriptstyle{a}$}}
\psfrag{b}[tl]{\scalebox{0.55}{$\scriptstyle{b}$}} 
\psfrag{g}[tl]{\scalebox{0.6}{$\scriptstyle{\gamma}$}} 
\psfrag{t}[tl]{\scalebox{0.6}{$\scriptstyle{c}$}} 
\psfrag{bn+1}[tl]{\scalebox{0.6}{$\scriptstyle{\beta_{n+1}}$}}
\psfrag{b1}[tl]{\scalebox{0.6}{$\scriptstyle{\beta_{1}}$}} 
%{\scalebox{0.9}{\includegraphics{cikis_yon}}}
\includegraphics[width=0.9\textwidth]{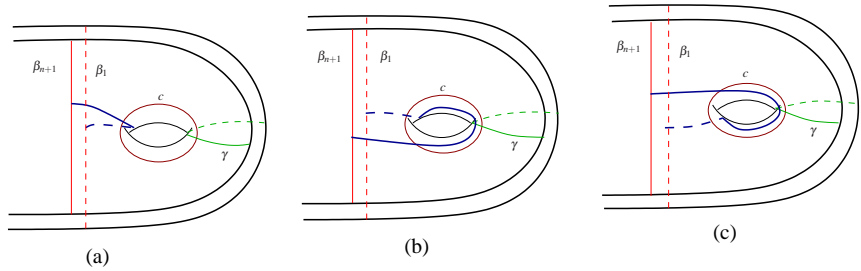}
\caption{ (a) Non-twist component, (b) Negative twist  component, (c) Positive twist  component.}\label{cikis_yon}
\end{figure}

%That is,
%
%\begin{enumerate}
%
%\item \emph{Above component}: End points on $\beta_i$ and $\beta_{i+1}$ intersecting with $\alpha_{2i-1}$. 
%
%\item \emph{Below component}: End points on $\beta_i$ and $\beta_{i+1}$ intersecting with $\alpha_{2i}$.
%
%\item \emph{Right loop component}: Both end points are on $\beta_i$ intersecting  with $\alpha_{2i-1}$ and $\alpha_{2i}$ (Figure~\ref{puncture_bilesen} (b)).
%
%\item \emph{Left loop component}: Both end points are on $\beta_{i+1}$ intersecting with $\alpha_{2i-1}$ and $\alpha_{2i}$ (Figure~\ref{puncture_bilesen} (a)).
%
%\item \emph{$c$ curve}: It bounds the genus of surface (Figure~\ref{genus_loop} (a)).
%
%\item \emph{Front genus component}: Both end points are on $\beta_{n+1}$  not intersecting with curve $c$ (Figure~\ref{genus_loop} (b)).
%
%\item \emph{Back genus component}: Both endpoints are on $\beta_{1}$ not intersecting with curve $c$  (Figure~\ref{genus_loop} (c)).
%
%\item \emph{Cutting component}: End points on $\beta_{1}$ and $\beta_{n+1}$ intersecting with curve $c$ (see Figure~\ref{cikis_yon}). There are $3$ types of such components:
%
%\begin{enumerate}
%
%\item  \emph{Non-twist cutting component}: See Figure~\ref{cikis_yon} (a)  
% 
%\item  \emph{Negative twist cutting component}: clockwise twist (See Figure~\ref{cikis_yon} (b))
%
%\item  \emph{Positive twist cutting component}: counterclockwise twist (See Figure~\ref{cikis_yon} ~(c))
%
%\end{enumerate}
%\end{enumerate}
\noindent A twisting component's twist number is the signed number of intersections with $\gamma$ curve.

\begin{remark}\label{not_c_cutting}
Since an integral lamination  $L \in \cL_{n}$ consists of simple closed curves that do not intersect each other, there can not be both  curve  $c$  and twisting components at the same time in the region $G$ (see Figure~\ref{paralel_cikis}). Also note that there are a uniform front genus and a uniform back genus component in the  region  $G$.
\begin{figure}[!ht]
\centering
\psfrag{a}[tl]{\scalebox{0.55}{$\scriptstyle{a}$}}
\psfrag{b}[tl]{\scalebox{0.55}{$\scriptstyle{b}$}} 
\psfrag{g}[tl]{\scalebox{0.6}{$\scriptstyle{\gamma}$}}
\psfrag{t}[tl]{\scalebox{0.6}{$\scriptstyle{c}$}}
\psfrag{bn+1}[tl]{\scalebox{0.6}{$\scriptstyle{\beta_{n+1}}$}}
\psfrag{b1}[tl]{\scalebox{0.6}{$\scriptstyle{\beta_{1}}$}}
%{\scalebox{0.9}{\includegraphics{paralel_cikis}}}
\includegraphics[width=0.5\textwidth]{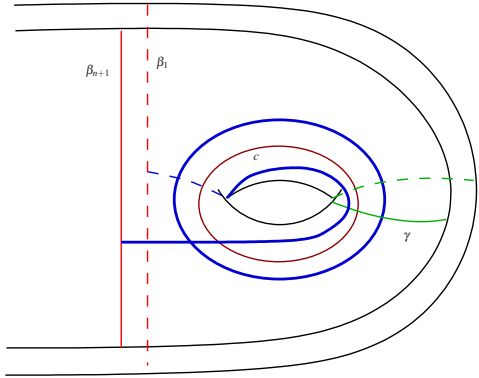}
\caption{ $L$ does not contain both  curve $c$ and twisting components at the same time. }\label{paralel_cikis}
\end{figure}
\end{remark}

\begin{remark}\label{num_cutting}
Note that the number  $c^{+}$  gives the number of twisting components.
\end{remark}

\begin{remark}\label{not_greater_1}
Since an integral lamination does not contain any self-intersections, directions of the twists  has to be the same. Also, in region $G$, 
the difference between the twist numbers of two different such components  can not be greater than $1$ (see Figure~\ref{fark_big_zero}).
\begin{figure}[!ht]
        \centering
        \psfrag{a}[tl]{\scalebox{0.55}{$\scriptstyle{a}$}}
        \psfrag{b}[tl]{\scalebox{0.55}{$\scriptstyle{b}$}}
        \psfrag{g}[tl]{\scalebox{0.6}{$\scriptstyle{\gamma}$}} 
        \psfrag{t}[tl]{\scalebox{0.6}{$\scriptstyle{c}$}}
        \psfrag{bn+1}[tl]{\scalebox{0.6}{$\scriptstyle{\beta_{n+1}}$}}
				\psfrag{b1}[tl]{\scalebox{0.6}{$\scriptstyle{\beta_{1}}$}} 
				%{\scalebox{0.9}{\includegraphics{fark_big_zero}}}
        \includegraphics[width=0.5\textwidth]{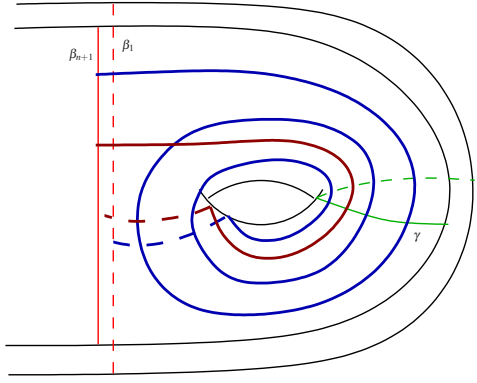}
        \caption{ If the difference between the twist numbers of two twisting components is greater than one, then they intersect.}\label{fark_big_zero}
        \end{figure}

If we denote the \emph{smaller twist number} by $t$ and the \emph{bigger twist number} by $t+1$, then the \emph{total twist number} ~$T$  in 
$G$ is the sum of the twist numbers of such components. Hence,  if the difference between twist numbers of any two twisting components is $0$, then  
$$T = tc^{+}.$$  
\noindent 
On the other hand, if the difference between twist numbers of any   two twisting components is $1$, then $$ T = m(t + 1) + (c^{+} - m)t $$  \noindent 
where $m \in \Z_{\geq0}$ is the number of twisting components with twist number  $t+1$, and  $ c^{+} - m $ is the number of  twisting components with twist number  $t$.
\end{remark}

Now, we calculate the path components  of $L$ in  $G$:
\begin{lemma}\label{cins}
Let $L$ be given with the intersection numbers  $(\alpha; \beta; \gamma; c)$, and the number of front genus components and the number of back genus components be $l$ and $l'$, respectively. Then,
\begin{eqnarray*}
l = \frac{\beta_{n+1} - c^{+}}{2} \quad \text{ and } \quad l' = \frac{\beta_{1} - c^{+}}{2}
\end{eqnarray*}
\end{lemma}

\begin{proof} 
$\beta_{n+1}$ intersects only with twisting (Figure~\ref{cikis_yon}) and front genus (Figure~\ref{genus_loop} (b)) components. Since $\beta_{n+1}$ intersects once with each twisting component and twice with each front genus component, $\beta_{n+1} = c^{+} + 2l$. From here, $l = \frac{\beta_{n+1} - c^{+}}{2}$ is derived. Similarly,  $\beta_{1}$ intersects only with twisting (Figure~\ref{cikis_yon}) and back genus (Figure~\ref{genus_loop} (c)) components. Since $\beta_{1}$  intersects once with each twisting component and twice with each back genus component, $\beta_{1} = c^{+} + 2l'$. Therefore, $l' = \frac{\beta_{1} - c^{+}}{2}$ is derived. 
\end{proof}

In the following theorem, we calculate the total twist number of twisting components:

\begin{lemma}\label{lem_total_twist}
Let $L$ be given with the intersection numbers  $(\alpha; \beta; \gamma; c)$, denoting the signed total twist number of twisting components by $T$. We have
\begin{align}\label{top_burgu}
|T|&= \left\{ \begin{array}{ll}
       0& \mbox{if $c^{+} = 0$},\\
         \gamma - \frac{\beta_{n+1} - c^{+}}{2} - \frac{\beta_{1} - c^{+}}{2} & \mbox{if $c^{+} \neq 0$.}   
         \end{array} \right.
\end{align}
\noindent
The sign of the negative twist component is $-1$ and the sign of the positive twist component is $1$. 
%When the cutting components twist in the positive direction, \textbf{sgn}($T$) ~$= +1$. When they twist in the negative direction, \textbf{sgn}($T$) ~$= -1$.
\end{lemma}

\begin{proof}
Let us denote the total twist number of twisting components of $L$ by $|T|$.  Note that the curve $\gamma$   intersects once with curve $c$ (Figure~\ref{genus_loop} (a)) 
and it intersects once with each front  and back genus components (Figures~\ref{genus_loop} (b) and (c), respectively).  Also, $\gamma$ intersects by the total number of twists of twisting components (Figure~\ref{cikis_yon}) with $L$. However, from Remark~\ref{not_c_cutting}, there can not be 
twists and  curve $c$ at the same time.  Therefore, when $c^{+} \neq 0$, we have 
\begin{equation}\label{gama_kesenler}
\gamma = l + l' + |T|
\end{equation}
\noindent where $l$, ~$l'$ and $|T|$ denote the number of front genus, back genus components and total twist number of twisting components, respectively.

From Lemma~\ref{cins},
\begin{equation*}
\gamma = \frac{\beta_{n+1} - c^{+}}{2} + \frac{\beta_{1} - c^{+}}{2} + |T|.
\end{equation*}
\noindent Hence,
\begin{equation*}
|T| = \gamma - \frac{\beta_{n+1} - c^{+}}{2} - \frac{\beta_{1} - c^{+}}{2}.
\end{equation*}
\end{proof}

By using the following theorem, we can calculate the number of  curves $c$ (Figure~\ref{genus_loop} (a)):
\begin{lemma}
Let $L$ be given with the intersection numbers  $(\alpha; \beta; \gamma; c)$. The number of  curves $c$ in $L$ is given by
\begin{align}\label{c_egri_number}
p(c)& = \left\{ \begin{array}{ll}
       \gamma - \frac{\beta_{n+1}}{2} - \frac{\beta_{1}}{2}& \mbox{if $c^{+} = 0$},\\
         0& \mbox{if $c^{+} \neq 0$.}   
         \end{array} \right.
\end{align}
\end{lemma}

\begin{proof}
Since $c^{+} = 0$, we can write $\gamma = l + l' + p(c)$. From Lemma~\ref{cins}, 
\begin{eqnarray*}
l = \frac{\beta_{n+1}}{2}  \quad \text{ and } \quad  l' = \frac{\beta_{1}}{2}.
\end{eqnarray*}
Hence, $p(c) = \gamma - \frac{\beta_{n+1}}{2} - \frac{\beta_{1}}{2}$ is derived.
\end{proof}

The twist numbers of each twisting component of an integral lamination whose intersection numbers are given are found by using Remark~\ref{not_greater_1} and Lemma~\ref{lem_total_twist}, which we find these twist numbers with the following lemma. 
\begin{lemma}\label{each_twist}
Let $L$ be given with the intersection numbers  $(\alpha; \beta; \gamma; c)$. Let $|T|$ and $m$ be the total twist number and the number of twisting components which has $t+1$ twists, respectively. In this case,
\begin{equation}
m \equiv |T| ~(mod ~c^{+})  \quad \text{ and } \quad t = \frac{|T| - m}{c^{+}}
\end{equation}
\noindent where $c^{+} \neq 0$.
\end{lemma}

\begin{proof}
From Remark~\ref{not_greater_1},
\begin{equation*}
|T| = m(t + 1) + (c^{+} - m)t.
\end{equation*}
\noindent From here, we have
\begin{equation*}\label{bolme_kurali}
|T| = m  +  tc^{+}.
\end{equation*}
\noindent Hence,
\begin{equation*}\label{mod}
m \equiv |T| ~(mod ~c^{+}) \quad \mbox{ and } \quad t = \frac{|T| - m}{c^{+}} 
\end{equation*}
\noindent are derived.
                                                                    
\end{proof}

\begin{remark}
The intersection numbers  $(\alpha; \beta; \gamma; c)$ might not always give an integral lamination. Because intersection numbers may not provide the conditions given in Lemma~\ref{equalities} or Lemma~\ref{odd_even}, and  the   triangle inequality in each region where is bounded by $\alpha_{2i-1}$, ~$\alpha_{2i}$, ~$\beta_{i}$ or by $\alpha_{2i-1}$, ~$\alpha_{2i}$, ~$\beta_{i+1}$.
\end{remark}

To illustrate, we can not construct an integral lamination having the intersection numbers $(1, 1, 1, 1, 1, 1; 0, 2, 0, 2; 2; 1)$. Because, according to Lemma~\ref{cins}, the numbers of front genus and back genus components are respectively
$$
l = \frac{\beta_{4} - c^{+}}{2} = \frac{2 - 1}{2} = \frac{1}{2} \notin \Z_{\geq 0} \quad \mbox{ and } \quad l' = \frac{\beta_{1} - c^{+}}{2} = \frac{0 - 1}{2} = -\frac{1}{2} \notin \Z_{\geq 0}.
$$
\noindent In such a case, any integral lamination can not be constructed as shown in Figure~\ref{not_surjec}.
\begin{figure}[!ht]
\centering
%{\scalebox{0.5}{\includegraphics{not_surjec}}}
\includegraphics[width=0.72\textwidth]{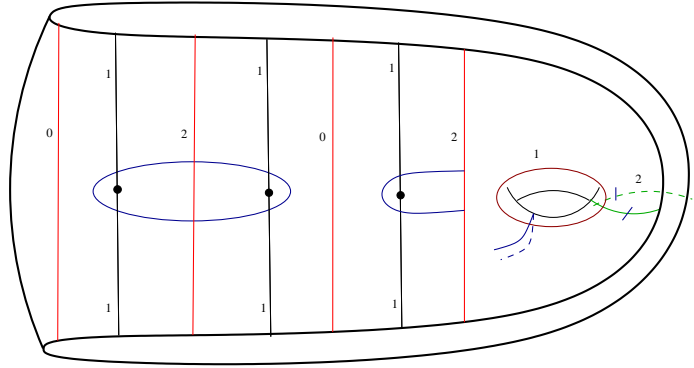}
\caption{$\alpha_{2i}\cup \alpha_{2i-1}$ and $\beta_{i}$  are each even, however $c$ is odd. }\label{not_surjec}
\end{figure}

\begin{remark}
Let the number of loop components in each region $U_i$  for $1 \leq i \leq n$ be denoted by $|b_i|$, where 
\begin{equation}\label{puncture_components}
b_{i} = \frac{\beta_{i} - \beta_{i+1}}{2}.
\end{equation}
If $b_i < 0$, loop component is called \emph{left}; if $b_i > 0$, loop component is called \emph{right} \cite{dynnikov02}.
\end{remark}

\begin{lemma}[\cite{yurttas11}]\label{equalities}
The following equalities hold for each $U_i$:

When there is a left loop component,
\begin{equation*}
\alpha_{2i} + \alpha_{2i-1} = \beta_{i+1}
\end{equation*}
\begin{equation*}
\alpha_{2i} + \alpha_{2i-1} - \beta_{i} = 2|b_{i}|,
\end{equation*}
\noindent when there is a right loop component,
\begin{equation*}
\alpha_{2i} + \alpha_{2i-1} = \beta_{i}
\end{equation*}
\begin{equation*}
\alpha_{2i} + \alpha_{2i-1} - \beta_{i+1} = 2|b_{i}|,
\end{equation*}
\noindent when there is no  loop components,
\begin{equation*}
\alpha_{2i} + \alpha_{2i-1} = \beta_{i} = \beta_{i+1}.
\end{equation*}
\end{lemma}

\begin{lemma}\label{odd_even}
Let $L$ be given with the intersection numbers  $(\alpha; \beta; \gamma; c)$. Then for each $1 \leq i \leq n$, ~$\beta_{i} - \beta_{i+1}$ and $\alpha_{2i} - \alpha_{2i-1} - c^{+}$ are even.
\end{lemma}

\begin{proof}
From Lemma~\ref{cins}, since $$\beta_{n+1} = c^{+} + 2l,$$ \noindent if $c^{+}$ is even (odd), $\beta_{n+1}$ is even (odd). Similarly, since $$\beta_{1} = c^{+} + 2l',$$ if $c^{+}$ is even (odd), $\beta_{1}$ is even (odd). Also, from \cite{dynnikov02}, the number of loop components is given by $$b_{i} = \frac{\beta_{i} - \beta_{i+1}}{2}   ~(1 \leq i \leq n).$$  \noindent From here, we can write
$$\beta_{i+1} = \beta_{i} - 2\sum_{j=1}^{i}b_{j}.$$ Therefore, if $c^{+}$ is even (odd), each $\beta_i  ~(1 \leq i \leq n+1)$ is even (odd).

From Lemma~\ref{equalities}, when there is right loop component, $\alpha_{2i} + \alpha_{2i-1} = \beta_{i}$; when there is left loop component, $\alpha_{2i} + \alpha_{2i-1} = \beta_{i+1}$. Hence,
when $c^{+}$ is even (odd), $\alpha_{2i} + \alpha_{2i-1}$ is even (odd).

Since when $c^{+}$ is even, $\alpha_{2i} + \alpha_{2i-1}$ is even and when $c^{+}$ is odd, $\alpha_{2i} + \alpha_{2i-1}$ is odd, $\alpha_{2i} + \alpha_{2i-1} - c^{+}$ is always even.

\end{proof}

\begin{lemma}[\cite{dynnikov02}]\label{lem_ab_bel}
Let $L \in \cL_{n}$ be given with the intersection numbers  $(\alpha; \beta; \gamma; c)$. For each $1 \leq i \leq n$, the number of above, $u_i^a$, and below, $u_i^b$,  components in $U_i$ can be found by $$u_{i}^{a} = \alpha_{2i-1} - |b_{i}| \quad \mbox{ and } \quad u_{i}^{b}  = \alpha_{2i} - |b_{i}|.$$ 
\end{lemma}

\begin{remark}
The  intersection numbers of two different integral laminations might be the same.
\end{remark}

For example, while the  intersection numbers of two integral laminations given in Figure~\ref{draw_int_lam} are $(2, 2, 2, 2, 2, 2; 2, 4, 2, 4; 4; 2)$, since the twisting components of integral lamination in Figure~\ref{draw_int_lam} (a) twist in the negative direction and the twisting components of integral lamination in Figure~\ref{draw_int_lam} (b) twist in the positive direction, these integral laminations are different. Therefore, intersection numbers can not give an injective function.
\begin{figure}[!ht]
\centering
%{\scalebox{0.38}{\includegraphics{draw_int_lam}}}
\includegraphics[width=1.0\textwidth]{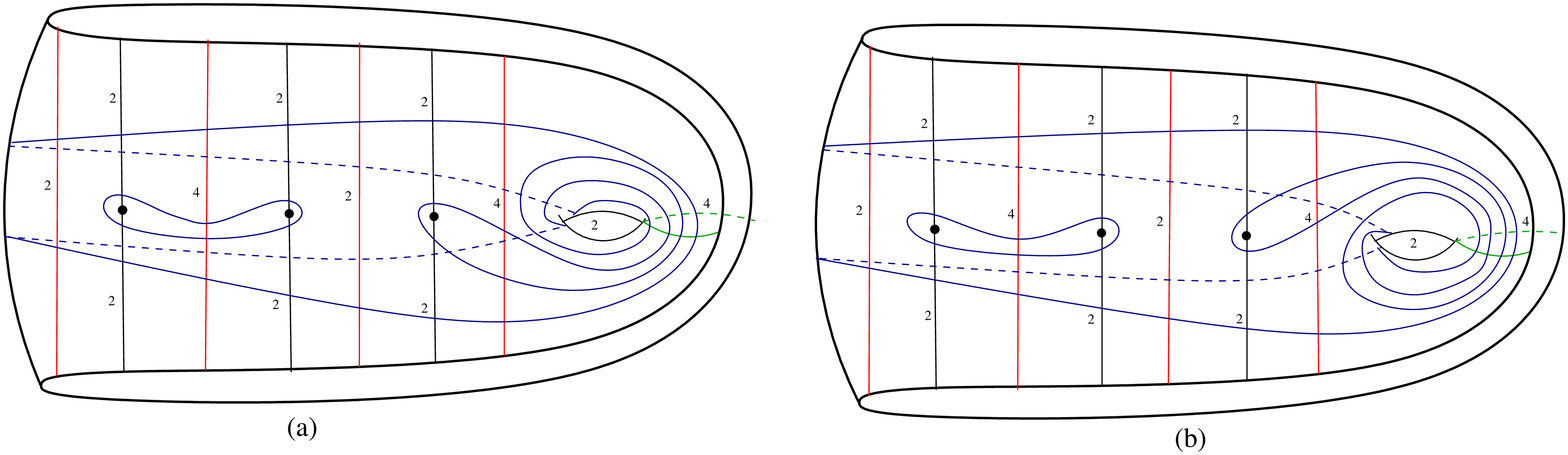}
\caption{ Two different integral laminations with the same intersection numbers}\label{draw_int_lam}
\end{figure}

\begin{remark}\label{being_injective}
Note that  two different integral laminations might have the same intersection numbers  (the directions of twisting components can be different). Therefore, we can derive an injective function from  intersection numbers  $(\alpha; \beta; \gamma; c)$ by giving a direction to the twists of twisting components. 
\end{remark}

Let  $m_{i} = \min (\alpha_{2i} - |b_{i}|, \alpha_{2i-1} - |b_{i}|)$ and set $2a_{i} = \alpha_{2i} - \alpha_{2i-1} - c^{+}$ for $1 \leq i \leq n$. By Lemma~\ref{odd_even}, $a_i$ is an integer. By similar calculations as in the way: disk case, we can derive the intersection number with $\alpha_i$ on $S_n$ in the following:

For each $1 \leq i \leq 2n$,

\begin{align}\label{alfa_inversion}
\alpha_{i}& = \left\{ \begin{array}{ll}
          \frac{2(-1)^{i}a_{\lceil i/2 \rceil} + (-1)^{i}c^{+} + \beta_{\lceil i/2 \rceil}}{2}& \mbox{if $b_{\lceil i/2 \rceil} \geq 0$},\\
         \frac{2(-1)^{i}a_{\lceil i/2 \rceil} + (-1)^{i}c^{+} + \beta_{(1 + \lceil i/2 \rceil)}}{2}& \mbox{if $b_{\lceil i/2 \rceil} \leq 0$}.   
         \end{array} \right.
\end{align}
\noindent where $\lceil x \rceil$ is the smallest integer greater than or equal to $x$.

Now, we derive the intersection number  with $\beta_i$ on $S_n$. Let $l$, ~$l'$ and $m_{i}$ ~$(1 \leq i \leq n)$ show the front genus number, back genus number and the minimum of above and below component numbers, respectively. Since $L$ can not contain a curve that bounds the boundary, at least one of $m_i$, ~$l$ or  $l'$ has to be $0$. There are two cases:

\textbf{Case 1:} Assume at least one of $m_i = 0$ for $ 1 \leq i \leq n $. In this case %(we refer the reader to \cite{meral19} and \cite{yurttas11} for more details):%
\begin{equation}\label{good_beta_son}
\beta_{n+1} = \max_{1 \leq k \leq n}\left[2\max(b_{k}, 0) + |2a_{k} + c^{+}| - 2\sum_{j=k}^{n}b_{j}\right] 
\end{equation}
and
\begin{equation}\label{good_big_bad}
  \beta_{n+1} \geq \max (c^{+}, c^{+} - 2\sum_{i=1}^{n}b_{i}). 
\end{equation}
\noindent An example for this case is depicted in Figure~\ref{beta_cift_ornek}.

\begin{figure}[!ht]
\centering
%{\scalebox{0.5}{\includegraphics{beta_cift_ornek}}}
\includegraphics[width=0.65\textwidth]{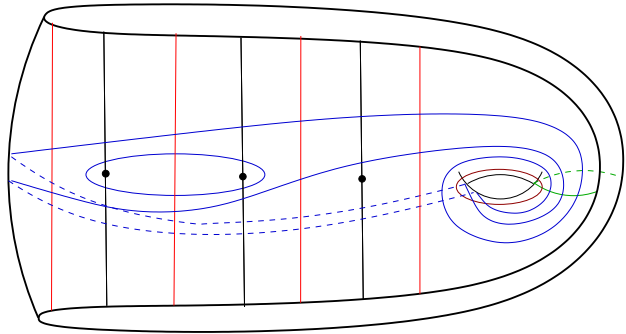}
\caption{ An integral lamination with $m_{i} \neq 0$ for $i = 1, 2$ }\label{beta_cift_ornek}
\end{figure}

\textbf{Case 2:} If $m_{i} \neq 0$ for any $ 1 \leq i \leq n $: In this case, an integral lamination contains curves whose each above and below component number are different from $0$ (see Figure~\ref{bad_start}).
\begin{figure}[!ht]
\centering
\psfrag{a}[tl]{\scalebox{0.55}{$\scriptstyle{\mbox{{(a)}}}$}}
\psfrag{b}[tl]{\scalebox{0.55}{$\scriptstyle{\mbox{{(b)}}}$}}
\includegraphics[width=1.0\textwidth]{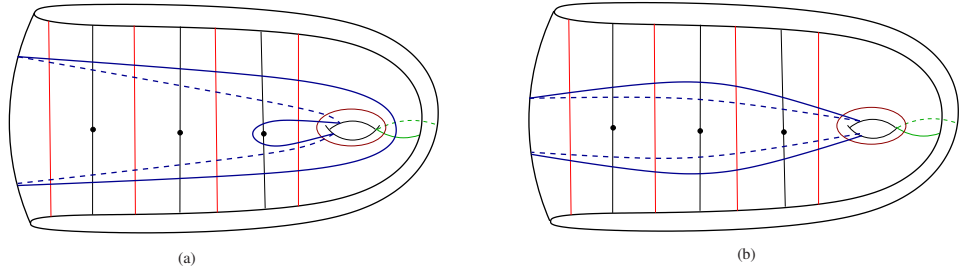}
%{\scalebox{0.35}{\includegraphics{bad_start}}}
\caption{ Integral laminations with each $m_{i}$ is different from $0$ }\label{bad_start}
\end{figure}
Also, at least one of the front genus or back genus component numbers must be $0$. Otherwise, this curve system contains curves parallel to the boundary as shown in the Figure~\ref{bo_paralel}.
\begin{figure}[!ht]
\centering
%\psfrag{a}[tl]{$\scriptstyle{\mbox{{\LARGE (a)}}}$}
%\psfrag{b}[tl]{$\scriptstyle{\mbox{{\LARGE (b)}}}$}
%{\scalebox{0.33}{\includegraphics{bo_paralel}}}
\includegraphics[width=0.65\textwidth]{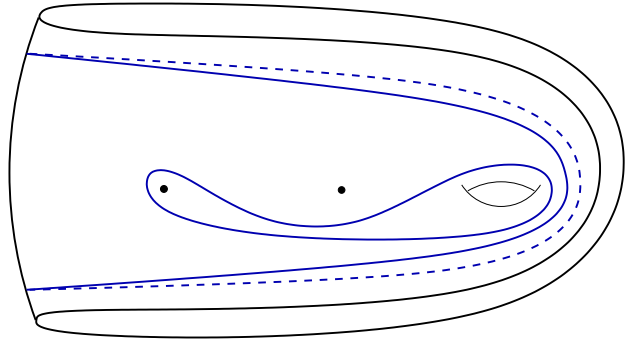}
\caption{ A curve system with a curve parallel to the boundary }\label{bo_paralel}
\end{figure}
Therefore, there are three possibilities: %(For more details see \cite{meral19})%:

\begin{enumerate}
\item[(i)] If $l = l' = 0$ and $\sum_{i=1}^{n}b_{i} = 0,$
\item[(ii)] If $l > 0$, $l' = 0$ and $\sum_{i=1}^{n}b_{i} < 0,$
\item[(iii)] If $l = 0$, $l' > 0$ and $\sum_{i=1}^{n}b_{i} > 0.$
\end{enumerate}
Combining cases (i), (ii) and (iii), we get 
\begin{align}\label{badbeta}
\beta_{n+1}& = \left\{ \begin{array}{ll}
          c^{+} - 2\sum_{i=1}^{n}b_{i} & \mbox{if $\sum_{i=1}^{n}b_{i} < 0$},\\
         c^{+} & \mbox{if $\sum_{i=1}^{n}b_{i} > 0$}, \\
				c^{+} & \mbox{if $\sum_{i=1}^{n}b_{i} = 0$}.
         \end{array} \right.
\end{align}
Since each $m_{i} > 0$  for  $(1 \leq i \leq n)$, we have 
\begin{equation}\label{bad_big_good}
\beta_{n+1} \geq \max_{1 \leq k \leq n}\left[2\max(b_{k}, 0) + |2a_{k} + c^{+}| - 2\sum_{j=k}^{n}b_{j}\right].
\end{equation}
In terms of brevity, let $$ \kappa:=  \max_{1 \leq k \leq n}\left[2\max(b_{k}, 0) + |2a_{k} + c^{+}| - 2\sum_{j=k}^{n}b_{j}\right]. $$
From inequalities (\ref{good_big_bad}) and (\ref{bad_big_good}), we have 
\begin{equation}\label{final_beta}
\beta_{n+1} = \max(c^{+}, c^{+} - 2\sum_{i=1}^{n}b_{i}, \kappa).
\end{equation}
From Equation~(\ref{puncture_components}), for each $1 \leq i \leq n$,
\begin{equation}\label{beta_inversion}
\beta_{i} = 2\sum_{j=i}^{n}b_{j} + \beta_{n+1}.
\end{equation}

Now, we derive the intersection number with $\gamma$ on $S_n$. Since each path component in the region $G$, 
except non-twist twisting components, intersects  with the arc $\gamma$ once, we have 
$$ 
\gamma = l + l' + p(c),
$$ 
\noindent 
if $p(c) \neq 0$, and when $p(c) = 0$, we have 
$$ 
\gamma = l + l' + |T|.
$$ 
\noindent 
Recall that from Equation~(\ref{copy}), $ p(c) = -c$.  Therefore,
\begin{align}\label{gama}
\gamma& = \left\{ \begin{array}{ll}
          |T| + l + l'& \mbox{if $c > 0$},\\
       |c| + l + l'& \mbox{if $c \leq 0$}.   
         \end{array} \right.
\end{align}

By Lemma~\ref{cins}, and  Equations~(\ref{final_beta}) and~(\ref{beta_inversion}), we have
\begin{align}\label{gama_inversion}
\gamma& = \left\{ \begin{array}{ll}
          |T| + \sum_{j=1}^{n}b_{j} + \max(c^{+}, c^{+} - 2\sum_{i=1}^{n}b_{i}, \kappa) - c^{+}& \mbox{if $c > 0$},\\
       |c| + \sum_{j=1}^{n}b_{j} + \max(c^{+}, c^{+} - 2\sum_{i=1}^{n}b_{i}, \kappa) - c^{+}& \mbox{if $c \leq 0$}.   
         \end{array} \right.
\end{align}

Above we have expressed the intersection numbers with arcs $\alpha_i$, $\beta_i$ and $\gamma$ in terms of $a_i$, $b_i$ and $T$.  
Note that, for an integral lamination,  $T = 0$ when $c \leq 0$.
Now, we can define \emph{Dynnikov coordinate system on} $S_n$ which bijectively coordinatizes the set  $\cL_n$ .

\begin{definition}
Let $\cV_{n} = \{(a; b; T; c): c \leq 0 \mbox{ and } T \neq 0\} \cup \{0\}.$ \emph{Dynnikov coordinate function} $\Phi : \cL_{n} \rightarrow \Z^{2n+2} \setminus \cV_{n}$ on $S_n$  is defined by 
\begin{equation*}
\Phi(L) = (a; b; T; c) = (a_{1}, \cdots, a_{n}; b_{1}, \cdots, b_{n}; T; c)
\end{equation*}
\noindent where for each $1 \leq i \leq n$,
\begin{equation}\label{Gen_Dyn_a_b_equ}
a_{i} = \frac{\alpha_{2i} - \alpha_{2i-1} - c^{+}}{2}, \quad ~b_{i} = \frac{\beta_{i} - \beta_{i+1}}{2} 
\end{equation}
\noindent and
\begin{align}\label{Dyn_twist_equ}
|T|&= \left\{ \begin{array}{ll}
       0& \mbox{if $c^{+} = 0$},\\
         \gamma - \frac{\beta_{n+1} - c^{+}}{2} - \frac{\beta_{1} - c^{+}}{2} & \mbox{if $c^{+} \neq 0.$}   
         \end{array} \right.
\end{align}
\end{definition}

\begin{example}
We calculate the Dynnikov coordinates of the integral lamination shown in Figure~\ref{ucgen_ornek}. 

Since $ (\alpha_{1}, \alpha_{2}, \alpha_{3}, \alpha_{4}, \alpha_{5}, \alpha_{6}; \beta_{1}, \beta_{2}, \beta_{3}, \beta_{4}; \gamma; c) = (4, 1, 3, 2, 4, 1; 3, 5, 5, 3; 3; 1),$ from Equations~(\ref{Gen_Dyn_a_b_equ})
\begin{equation*}
a_{1} = \frac{\alpha_{2} - \alpha_{1} - c^{+}}{2} = \frac{1 - 4 - 1}{2} = -2,
\end{equation*}

\begin{equation*}
a_{2} = \frac{\alpha_{4} - \alpha_{3} - c^{+}}{2} = \frac{2 - 3 - 1}{2} = -1,
\end{equation*}

\begin{equation*}
a_{3} = \frac{\alpha_{6} - \alpha_{5} - c^{+}}{2} = \frac{1 - 4 - 1}{2} = -2,
\end{equation*}

\begin{equation*}
b_{1} = \frac{\beta_{1} - \beta_{2}}{2} = \frac{3 - 5}{2} = -1,
\end{equation*}

\begin{equation*}
b_{2} = \frac{\beta_{2} - \beta_{3}}{2} = \frac{5 - 5}{2} = 0,
\end{equation*}

\begin{equation*}
b_{3} = \frac{\beta_{3} - \beta_{4}}{2} = \frac{5 - 3}{2} = 1.
\end{equation*}
Also, since $c^{+} = \max(c, 0) = \max(1, 0) = 1$, from Equation~(\ref{Dyn_twist_equ}), $$ |T| = \gamma - \frac{\beta_{4} - c^{+}}{2} - \frac{\beta_{1} - c^{+}}{2} = 3 - \frac{3 - 1}{2} - \frac{3 - 1}{2} = 1.$$ 
Since the twisting component twists in the negative direction, we derive $T = -1$.  Hence, we find 
$$\Phi(L) = (-2, -1, -2; -1, 0, 1; -1; 1).$$
\end{example}

The following theorem Theorem~\ref{inversion_formul} gives the \emph{inversion of Dynnikov coordinate function} on $S_n$:

\begin{theorem}\label{inversion_formul}
Let $(a; b; T; c) \in \Z^{2n+2} \setminus \cV_{n}$. Then, the vector $(a; b; T; c)$ corresponds to one and only one integral lamination $L \in \cL_n$ whose  intersection numbers are given by
\begin{equation}\label{beta_inv}
\beta_{i} = 2\sum_{j=i}^{n}b_{j} + \max(c^{+}, c^{+} - 2\sum_{i=1}^{n}b_{i}, \kappa), \quad ~\beta_{n+1} = \max(c^{+}, c^{+} - 2\sum_{i=1}^{n}b_{i}, \kappa)
\end{equation}

\begin{align}\label{alfa_inv}
\alpha_{i}& = \left\{ \begin{array}{ll}
          \frac{2(-1)^{i}a_{\lceil i/2 \rceil} + (-1)^{i}c^{+} + \beta_{\lceil i/2 \rceil}}{2}& \mbox{if $b_{\lceil i/2 \rceil} \geq 0$},\\
         \frac{2(-1)^{i}a_{\lceil i/2 \rceil} + (-1)^{i}c^{+} + \beta_{(1+\lceil i/2 \rceil)}}{2}& \mbox{if $b_{\lceil i/2 \rceil} \leq 0$}.   
         \end{array} \right.
\end{align}
\begin{center}
and
\end{center}
\begin{align}\label{gama_inv}
\gamma& = \left\{ \begin{array}{ll}
          |T| + \sum_{j=1}^{n}b_{j} + \max(c^{+}, c^{+} - 2\sum_{i=1}^{n}b_{i}, \kappa) - c^{+}& \mbox{if $c > 0$},\\
       |c| + \sum_{j=1}^{n}b_{j} + \max(c^{+}, c^{+} - 2\sum_{i=1}^{n}b_{i}, \kappa) - c^{+}& \mbox{if $c \leq 0$}.   
         \end{array} \right.
\end{align}
where 
\begin{equation*}
\kappa = \max_{1 \leq k \leq n}\left[2\max(b_{k}, 0) + |2a_{k} + c^{+}| - 2\sum_{j=k}^{n}b_{j}\right].
\end{equation*}

\end{theorem}

\begin{proof}
Let $L \in \cL_n$ be an integral lamination whose Dynnikov coordinates are $\Phi(L) = (a; b; T; c)$. ~Firstly, we shall show that Dynnikov function \linebreak $\Phi : \cL_{n} \rightarrow \Z^{2n+2} \setminus \cV_{n}$ is injective. It has been shown in this paper that the intersection numbers corresponding to the minimal representative $L \in \cL$  are as given in equations (\ref{beta_inv}), (\ref{alfa_inv}) and (\ref{gama_inv}). Then, the numbers of above, below, right loop or left loop components in each region $U_i$, the numbers of curves $c$, front genus, back genus, twisting components, the total twist of twisting components and the number of twists of each twisting component, the direction of these twists in the region $G$ are calculated as given in above, and therefore as indicated in Remark~\ref{being_injective}, the path components in regions $U_i$ and $G$ can be combined uniquely up to isotopy by giving a direction to the twists of twisting components. Hence, $\Phi $ is injective.

Now, we see that the function $\Phi \colon \cL_{n} \rightarrow \Z^{2n+2} \setminus \cV_{n}$ is surjective. Let \linebreak $(a; b; T; c)\in \Z^{2n+2} \setminus \cV_{n}$. 
We shall show that the intersection numbers $(\alpha; \beta; \gamma; c)$ defined by Equations (\ref{beta_inv}), (\ref{alfa_inv}) and (\ref{gama_inv}) correspond to an 
integral lamination $L \in \cL_n$ such that $\Phi(L) = (a; b; T; c)$. First of all, it is easy to see that  an integral lamination $L$ with intersection numbers $(\alpha; \beta; \gamma; c)$  
should satisfy $\Phi(L) = (a; b; T; c)$.  To get an integral lamination, we draw non-intersecting path components in each region and join them together.
\end{proof}

\begin{example}
Let the Dynnikov coordinates of the integral lamination $L \in \cL_{3}$ on $S_3$ be $ \Phi(L) = (a; b; T; c) = (-2, -2, -1; 0, -1, -1; -5; 3)$. We find the intersection numbers corresponding to the minimal representative $L$.

Since $c = 3 > 0$, $p(c) = 0$ and $c^{+} = 3$. From Theorem~\ref{inversion_formul}, the intersection numbers $\alpha$, $\beta$ and $\gamma$  are found as the following:

When we place the given Dynnikov coordinates  to the equation
\begin{equation*}
\kappa = \max_{1 \leq k \leq 3}\left[2\max(b_{k}, 0) + |2a_{k} + c^{+}| - 2\sum_{j=k}^{3}b_{j}\right],
\end{equation*}
we find $\kappa = 5$. From here,
\begin{align*}
\beta_{4} &= \max(c^{+}, c^{+} - 2\sum_{i=1}^{3}b_{i}, \kappa)\\
&= \max(3, 3 - 2(0 - 1 - 1), 5) = 7,
\end{align*}
that is, $\beta_{4} = 7$. From Equation~(\ref{beta_inversion}), we derive 
\begin{equation*}
\beta_{1} = 2(b_{1} + b_{2} + b_{3}) + \beta_{4} = 2(0 - 1 - 1) + 7 = 3.
\end{equation*}
By the similar calculations, we have $\beta_{2} = 3$ and $\beta_{3} = 5$. From Equation~(\ref{alfa_inv}), since $b_1 = 0$, $\alpha_1 = 2$ and  $\alpha_2 = 1$. Since $b_2 < 0$, $\alpha_3 = 3$ and  $\alpha_4 = 2$. Since $b_3 < 0$, $\alpha_5 = 3$ and  $\alpha_6 = 4$. Since $c > 0$, from Equation~(\ref{gama_inv}), 
\begin{align*}
\gamma &= |T| + \sum_{j=1}^{3}b_{j} + \max(c^{+}, c^{+} - 2\sum_{i=1}^{3}b_{i}, \kappa) - c^{+} \\
 &=  |T| + b_{1} + b_{2} + b_{3} + \max(c^{+}, c^{+} - 2(b_{1} + b_{2} + b_{3}), \kappa) - c^{+}\\
&=  5 + 0 - 1 - 1  + \max(3, 3 - 2(0 - 1 - 1), 5) - 3 = 7.
\end{align*}

Now, we calculate the numbers of path components in the regions $G$ and $U_i$.  Since $c^{+} = 3$ and $T = -5$, there are $3$ twisting components and the total number of twists is $5$ (see Remark~\ref{num_cutting} and Lemma~\ref{lem_total_twist}), and twisting components twist in the negative direction.  By Lemma~\ref{each_twist}, there are $2$ twisting components, which each twisting component does $2$ twists, and there is $1$ twisting component doing $1$ twist. According to Lemma~\ref{cins}, 
$$l = \frac{\beta_{4} - c^{+}}{2} = \frac{7 - 3}{2} = 2 \quad \mbox{ and } l' = \frac{\beta_{1} - c^{+}}{2} = \frac{3 - 3}{2} = 0.$$  That is, there are $2$ front genus components, however there is not any back genus component. 

 From Equation~(\ref{puncture_components}),
\begin{equation*}
b_{1} = \frac{\beta_{1} - \beta_{2}}{2} = \frac{3 - 3}{2} = 0.
\end{equation*}
Similarly, we have $b_2 = -1$ and $b_3 = -1$. Namely, there are no  loop components in region $U_1$, ~$1$ left loop component in region $U_2$ and $1$ left loop component in region $U_3$.

The numbers of above and below components in each $U_i$ are:

\begin{equation*}
u_{1}^{a} = \alpha_{1} - |b_{1}| = 2 - 0 = 2  \quad  ~\mbox{ and }~  \quad u_{1}^{b} = \alpha_{2} - |b_{1}| = 1 - 0 = 1.
\end{equation*}
That is, there are $2$ above components and $1$ below component in $U_1$.  By similar calculations, we find $2$ above and $1$ below components in $U_2$ and  $2$ above and $3$ below components in $U_3$.

The integral lamination $L$ in Figure~\ref{bad_exp_2} is derived uniquely by combining the calculated components  up to isotopy. 
\begin{figure}[!ht]
\centering
%{\scalebox{0.5}{\includegraphics{bad_exp_2}}}
\includegraphics[width=0.58\textwidth]{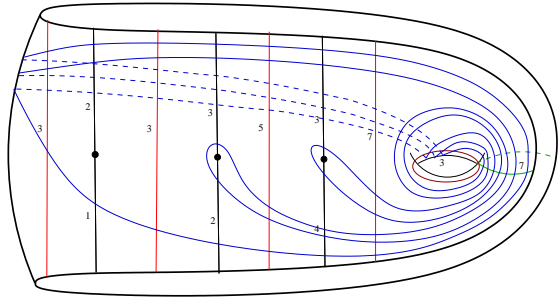}
\caption{ $ \Phi(L) = (a; b; T; c) = (-2, -2, -1; 0, -1, -1; -5; 3)$ }\label{bad_exp_2}
\end{figure}
\end{example}

\begin{remark}
The Dynnikov coordinates for integral laminations on $S_n$ obtained in this paper  can be extended in a natural way to Dynnikov coordinates of measured foliations on $S_n$. 
\end{remark}

\textbf{Acknowledgements:} The author would like to thank her advisor Saadet \"{O}yk\"{u} Yurtta\c{s} and co-advisor Semra Pamuk and point out that the results in this paper are part of author’s PhD thesis.

%%%%%%%%%%%%%%%%%%%%%%%%%%%%%%%%%%%%%%%%%%%%%%%%%%%%%%%%%%%%%%%%%%%%%%%%%%%%%%%%%%%%%%%%%%%%%%%%%%%%%%%%%%%%%%%%%%%%%%%%%%%%%%%%%%%
\bibliographystyle{amsplain}
\providecommand{\bysame}{\leavevmode\hbox
to3em{\hrulefill}\thinspace}
\providecommand{\MR}{\relax\ifhmode\unskip\space\fi MR }
% \MRhref is called by the amsart/book/proc definition of \MR.
\providecommand{\MRhref}[2]{%
  \href{http://www.ams.org/mathscinet-getitem?mr=#1}{#2}
} \providecommand{\href}[2]{#2}

\end{document}